\documentclass[11pt]{article}
\usepackage{amsmath}
\usepackage{amssymb}
\usepackage{amsthm}
\usepackage{enumerate}
\usepackage[colorlinks,draft=false]{hyperref}
\usepackage{enumitem}
\usepackage[margin=1in]{geometry}
\usepackage{stmaryrd}
\usepackage{pgfplots, pgfplotstable, booktabs, colortbl, array,multirow}
\usepgfplotslibrary{groupplots}
\pgfplotsset{compat=newest}

\ifdim\overfullrule>0pt 
  \usepackage{environ}

  \NewEnviron{tikzpicture}{%
    \begin{pgfpicture}
    \pgfpathrectanglecorners{\pgfpointorigin}{\pgfpoint{3cm}{3cm}}%
     \pgfusepath{stroke}\end{pgfpicture}%
  }
\fi

\usepackage{framed}

\theoremstyle{plain}
\newtheorem{theorem}{Theorem}[section]
\newtheorem{lemma}[theorem]{Lemma}
\newtheorem{proposition}[theorem]{Proposition}

\theoremstyle{definition}

\newtheorem{remark}{Remark}[section]

\DeclareMathOperator{\dv}{div}
\DeclareMathOperator{\curl}{curl}
\DeclareMathOperator{\grad}{grad}

\newcommand{\textoverline}[1]{$\overline{\mbox{#1}}$}

\newcommand{\plotsnapshots}[1]{
\includegraphics[scale=0.2,trim=16in 0in 15.5in 0in,clip=true]{#1t0.png}
\includegraphics[scale=0.2,trim=16in 0in 15.5in 0in,clip=true]{#1t1.png}
\includegraphics[scale=0.2,trim=16in 0in 15.5in 0in,clip=true]{#1t2.png}
\includegraphics[scale=0.2,trim=16in 0in 15.5in 0in,clip=true]{#1t3.png}
\includegraphics[scale=0.2,trim=16in 0in 15.5in 0in,clip=true]{#1t4.png}
\includegraphics[scale=0.2,trim=16in 0in 15.5in 0in,clip=true]{#1t5.png}
}

\newcommand{\plotinvariants}[1]{
\addplot[blue,thick] table [x expr=\coordindex*0.005, y expr=abs(\thisrowno{0})]{#1};
\addplot[red,thick] table [x expr=\coordindex*0.005, y expr=abs(\thisrowno{1})]{#1};
\addplot[black,thick] table [x expr=\coordindex*0.005, y expr=abs(\thisrowno{2})]{#1};
\addplot[blue,dashed,thick] table [x expr=\coordindex*0.005, y expr=abs(\thisrowno{3})]{#1};
}

\begin{document}

\title{A structure-preserving finite element method for compressible ideal and resistive MHD}

\author{Evan S. Gawlik\thanks{\noindent Department of Mathematics,  University of Hawai`i at M\textoverline{a}noa, \href{egawlik@hawaii.edu}{egawlik@hawaii.edu}} \; and \; Fran\c{c}ois Gay-Balmaz\thanks{\noindent CNRS - LMD, Ecole Normale Sup\'erieure, \href{francois.gay-balmaz@lmd.ens.fr}{francois.gay-balmaz@lmd.ens.fr}}}

\date{}

\maketitle

\begin{abstract}
We construct a structure-preserving finite element method and time-stepping scheme for compressible barotropic magnetohydrodynamics (MHD) both in the ideal and resistive cases, and in the presence of viscosity. The method is deduced from the geometric variational formulation of the equations. It preserves the balance laws governing the evolution of total energy and magnetic helicity, and preserves mass and the constraint $ \operatorname{div}B = 0$ to machine precision, both at the spatially and temporally discrete levels. In particular, conservation of energy and magnetic helicity hold at the discrete levels in the ideal case. It is observed that cross helicity is well conserved in our simulation in the ideal case.
\end{abstract}

\section{Introduction}

In this paper we develop a structure-preserving finite element method for the compressible barotropic MHD equations with viscosity and resistivity on a bounded domain $ \Omega \subset \mathbb{R} ^d$, $d \in \{2,3\}$. These equations seek a velocity field $u$, density $ \rho  $, and magnetic field $B$ such that
\begin{align}
\rho(\partial_t u + u \cdot \nabla u) - \operatorname{curl} B \times B &= -\nabla p + \mu \Delta u + ( \lambda + \mu ) \nabla \operatorname{div}u  & \text{ in } \Omega \times (0,T), \label{velocity0} \\
\partial_t B -  \operatorname{curl}(u \times B) &= - \nu \operatorname{curl}\operatorname{curl}B, & \text{ in } \Omega \times (0,T), \label{magnetic0} \\
\partial_t \rho + \dv (\rho u) &= 0, & \text{ in } \Omega \times (0,T), \label{density0} \\
\dv B &= 0, & \text{ in } \Omega \times (0,T), \label{incompressible0} \\
u = B \cdot  n = \curl B \times n &= 0, & \text{ on } \partial\Omega \times (0,T), \label{BC} \\
u(0) = u_0, \, B(0) = B_0, \,  \rho(0) &= \rho_0, & \text{ in } \Omega, \label{IC}
\end{align}
where $p=p( \rho  )$ is the pressure, $ \mu $ and $ \lambda $ are the fluid viscosity coefficients satisfying $ \mu >0$ and $2 \mu + 3 \lambda \geq 0$, and $ \nu >0$ is the resistivity coefficient.

The case $ \mu = \lambda = \nu =0$ corresponds to ideal non-viscous barotropic MHD, for which the boundary conditions \eqref{BC} are replaced by $u \cdot n|_{ \partial \Omega }= B \cdot n|_{ \partial \Omega }=0$.


Much of the literature on structure-preserving methods in MHD simulation has focused on the incompressible and ideal case, with constant density \cite{GaMuPaMaDe2011,HiLiMaZh2018,HuLeXu2020,HuMaXu2017,KrMa2017,LiWa2001,HuXu2019} and with variable density \cite{GaGB2021}. These methods have succeeded in preserving at the discrete levels several invariants and constraints of the continuous system. For instance, in \cite{GaGB2021} a finite element method was proposed which preserves energy, cross-helicity (when the fluid density is constant), magnetic helicity, mass, total squared density, pointwise incompressibility, and the constraint $ \operatorname{div}B=0$  to machine precision, both at the spatially and temporally discrete levels. Little attention has been paid to the development of structure-preserving methods for MHD in the compressible ideal or resistive case. For instance, in the resistive case, energy, magnetic helicity, and cross-helicity are not preserved and their evolution is governed by balance laws showing the impact of resistivity on the dynamics of these quantities. In order to accurately simulate the effect of resistivity in simulations of compressible MHD, it is highly desirable to exactly reproduce these laws at the discrete level. The discrete conservation laws for these quantities are then automatically satisfied in the ideal case, which extend similar properties obtained earlier in the incompressible setting.

In this paper, we construct a structure-preserving finite element method and time-stepping scheme for the compressible MHD system \eqref{velocity0}--\eqref{IC}. The method is deduced from the geometric variational formulation of the equations arising from the Hamilton principle on the diffeomorphism group of fluid motion. It preserves the balance laws governing the evolution of total energy and magnetic helicity, and preserves mass and the constraint $ \operatorname{div}B = 0$ to machine precision, both at the spatially and temporally discrete levels. In particular, conservation of energy and magnetic helicity hold at the discrete levels in the ideal case.

The approach we develop in this paper is built on our earlier work on conservative methods for compressible fluids \cite{GaGB2020} and for incompressible MHD with variable density in \cite{GaGB2019,GaGB2021}.  Two notable differences that arise in the viscous, resistive, compressible setting are the change in boundary conditions for the velocity and magnetic fields, and the fact that the magnetic field is not advected as a vector field when the fluid is compressible; that is, $\curl(B \times u)$ does not coincide with the Lie derivative of the vector field $B$ along $u$ when $\dv u \neq 0$.





\section{Geometric variational formulation for MHD}

In this section we review the Hamilton principle for ideal MHD as well as the associated Euler-Poincar\'e variational formulation. We then extend the resulting form of equations to include viscosity and resistivity and examine how the balance of energy, magnetic helicity, and cross-helicity emerge from this formulation.

\paragraph{Lagrangian variational formulation for ideal MHD.} Assume that the fluid moves in a compact domain $ \Omega \subset \mathbb{R} ^3$ with smooth boundary. We denote by $\operatorname{Diff}( \Omega )$ the group of diffeomorphisms of $ \Omega $ and by $ \varphi :[0,T] \rightarrow \operatorname{Diff}( \Omega )$ the fluid flow. The associated motion of a fluid particle with label $X \in \Omega $ is $x= \varphi (t,X)$.

\medskip

When $ \nu =0$, the equation for the magnetic field reduces to $ \partial _tB -  \operatorname{curl} (u \times B)=0$, which can be equivalently rewritten in geometric terms as $\partial _t (B \cdot {\rm d}s)+ \pounds _u ( B \cdot {\rm d} s)=0$ with $ \pounds _u(B \cdot {\rm d}s)$ the Lie derivative of the closed $2$-form $B \cdot {\rm d}s$.
Consequently, from the properties of Lie derivatives, the time evolution of the magnetic field is given by the push-forward operation on 2-forms as
\begin{equation}\label{B_frozen} 
B(t) \cdot {\rm d}s= \varphi (t)_* ( \mathcal{B} _0 \cdot {\rm d} S)
\end{equation}
for some time independent reference magnetic field $ \mathcal{B} _0(X)$. This describes the fact that the magnetic field is frozen in the flow. Similarly, from the continuity equation $ \partial  _t \rho  + \operatorname{div}( \rho  u)=0$, the evolution of the mass density is given by the push-forward operation on 3-forms as
\[
\rho  (t) {\rm d}^3 x = \varphi (t)_* ( \varrho _0 {\rm d}^3X),
\]
for some time independent reference mass density $ \varrho _0(X)$.

\medskip

From these considerations, if follows that the ideal MHD motion is completely characterized by the fluid flow $ \varphi (t)  \in  \operatorname{Diff}( \Omega )$ and the given reference fields $ \varrho _0$ and $ \mathcal{B} _0$. The Hamilton principle for this system reads
\begin{equation}\label{HP_MHD} 
\delta \int_0^T L( \varphi , \partial _t \varphi , \varrho_0, \mathcal{B}_0) {\rm d}t=0,
\end{equation} 
with respect to variations $ \delta \varphi $ vanishing at $t=0,T$, and yields the equations of motion in Lagrangian coordinates. In \eqref{HP_MHD} the Lagrangian function $L$ depends on the fluid flow $ \varphi (t)$ and its time derivative $ \partial _t\varphi (t)$ forming an element $( \varphi , \partial _t \varphi )$  in the tangent bundle $T \operatorname{Diff}( \Omega ) $ to $\operatorname{Diff}( \Omega ) $, and also parametrically on the given $ \varrho _0$, $ \mathcal{B} _0$. From the relabelling symmetries, $L$ must be invariant under the subgroup $ \operatorname{Diff}( \Omega )_{ \varrho _0, \mathcal{B}_0 } \subset \operatorname{Diff}( \Omega ) $ of diffeomorphisms that preserve $ \varrho _0$ and $ \mathcal{B} _0$, i.e., diffeomorphisms $ \psi \in \operatorname{Diff}( \Omega )$ such that 
\[
\psi ^*( \varrho _0 {\rm d}^3X) = \varrho _0 \quad\text{and}\quad \psi ^* ( \mathcal{B} _0 \cdot {\rm d}S) = \mathcal{B} _0 \cdot {\rm d}S,
\]
i.e., we have
\begin{equation}\label{invariance} 
L( \varphi \circ \psi , \partial _t( \varphi \circ \psi ), \varrho _0, \mathcal{B} _0)= L( \varphi , \partial _t \varphi , \varrho_0, \mathcal{B}_0), \quad \forall\; \psi \in  \operatorname{Diff}( \Omega )_{ \varrho _0, \mathcal{B}_0 } \subset \operatorname{Diff}( \Omega ) .
\end{equation} 
From this invariance, $L$ can be written in terms of Eulerian variables as
\begin{equation}\label{L_ell} 
L( \varphi , \partial _t \varphi , \varrho_0, \mathcal{B}_0)=\ell(u, \rho  , B)
\end{equation} 
where
\begin{equation} \label{Eul_variables}
u = \partial _t \varphi \circ \varphi ^{-1} , \qquad \rho \, {\rm d}^3 x = \varphi _* ( \varrho _0 {\rm d}^3 X), \qquad B \cdot {\rm d s} = \varphi _* ( \mathcal{B} _0 \cdot {\rm d} S),
\end{equation}
thereby yielding the symmetry reduced Lagrangian $\ell(u, \rho  , B)$ in the Eulerian description.
In terms of $\ell$, Hamilton's principle \eqref{HP_MHD} reads
\begin{equation}\label{EP_MHD} 
\delta \int_0^T\ell(u, \rho  , B) {\rm d}t=0,
\end{equation} 
with respect to variations of the form 
\begin{equation}\label{EP_constraints} 
\delta u = \partial _t v+ \pounds _uv, \qquad \delta \rho  = - \operatorname{div}( \rho  v), \qquad \delta B = \operatorname{curl} ( v \times B) 
\end{equation} 
where $v:[0,T] \rightarrow \mathfrak{X} ( \Omega )$ is an arbitrary time dependent vector field with $v(0)=v(T)=0$ and $ \pounds _uv=[u,v]$ is the Lie derivative of vector fields. Here $ \mathfrak{X} ( \Omega )$ denotes the space of vector fields $u$ on $ \Omega $ with $u \cdot n=0$ on $ \partial \Omega $, viewed as the Lie algebra of $ \operatorname{Diff}( \Omega )$. We recall that $B \cdot n=0$ on $ \partial \Omega $, a condition that is preserved by the evolution \eqref{B_frozen}. The passing from \eqref{HP_MHD} to \eqref{EP_MHD}  is a special instance of the process of Euler-Poincar\'e reduction for invariant systems on Lie groups, see \cite{HoMaRa1998}. A direct application of \eqref{EP_MHD}--\eqref{EP_constraints} yields the fluid momentum equations in the form
\begin{equation}\label{velocity} 
\left\langle \partial_t \frac{\delta \ell}{\delta u} , v \right\rangle + a\left(\frac{\delta \ell}{\delta u}, u, v\right)  + b\left( \frac{\delta \ell}{\delta \rho},\rho,v \right) + c\left(\frac{\delta \ell}{\delta B},B,v\right)=0,
\end{equation} 
for all $v$ with $v \cdot n=0$, with the trilinear forms
\begin{align*} 
a(w,u,v)&= -\int_ \Omega  w \cdot [u,v]\,{\rm d}x \\
b( \sigma , \rho  ,v)&= - \int_ \Omega  \rho    \nabla \sigma \cdot v \,{\rm d}x\\
c(C,B,v)&= \int_ \Omega  C \cdot \operatorname{curl}   (B \times v) \,{\rm d}x.
\end{align*}
The equations for $ \rho  $ and $B$ follow from their definition in \eqref{Eul_variables}, which are expressed in terms of $b$ and $c$ as
\begin{align}
\langle \partial_t \rho, \sigma \rangle + b(\sigma,\rho,u) &= 0, \quad \forall\; \sigma    \label{density} \\
\langle \partial_t B, C \rangle + c(C,B,u) &= 0, \quad \forall\; C, \quad C \cdot n|_{ \partial \Omega }=0.\label{magnetic}
\end{align}
Equation \eqref{velocity} yields the general Euler-Poincar\'e form of the equations for arbitrary Lagrangian $\ell(u, \rho  , B)$ as
\begin{equation}\label{EP_equations} 
\partial _t \frac{\delta \ell}{\delta u} + \pounds _u \frac{\delta \ell}{\delta u} = \rho  \nabla \frac{\delta \ell}{\delta \rho  }+ B \times \operatorname{curl}\frac{\delta \ell}{\delta B}  
\end{equation} 
where in the second term we employed the notation $\pounds _u m=  \operatorname{curl} m\times u + \nabla ( u \cdot m) + m \operatorname{div}u$.

\medskip 

The Lagrangian for barotropic MHD is
\begin{equation}\label{barotropic_ell} 
\ell(u, \rho  , B)= \int_ \Omega  \Big[\frac{1}{2} \rho  |u|^2 - \epsilon ( \rho  ) - \frac{1}{2} |B| ^2 \Big]\,{\rm d}x
\end{equation} 
with $ \epsilon ( \rho  )$ the energy density. Using
\[
\frac{\delta \ell}{\delta u}= \rho  u, \qquad \frac{\delta \ell}{\delta \rho  } = \frac{1}{2} |u| ^2 - \frac{\partial \epsilon }{\partial \rho  } , \qquad \frac{\delta  \ell}{\delta  B} = - B
\]
in \eqref{EP_equations} yields the barotropic MHD equations \eqref{velocity0} with $ \mu = \lambda =0$.

\medskip 

Extension to full compressible ideal MHD subject to gravitational and Coriolis forces is easily achieved by including the entropy density $s$ in the variational formulation and considering the Lagrangian function
\begin{equation} \label{baroclinic_ell}
\ell(u, \rho ,s , B)= \int_ \Omega  \Big[\frac{1}{2} \rho  |u|^2 + \rho  R \cdot u - \epsilon ( \rho ,s ) - \rho  \phi - \frac{1}{2} |B| ^2\Big] \,{\rm d} x,
\end{equation}
with $ \phi $ the gravitational potential and a vector field $R$ such that $ \operatorname{curl}R= 2 \omega $ with $ \omega $ the angular velocity of the fluid domain.

\paragraph{Viscous and resistive MHD.} Viscosity and resistivity are included in the formulation \eqref{velocity}--\eqref{magnetic}  
by defining the symmetric bilinear forms
\begin{equation}\label{bilinear_form} 
d(u,v)= -  \int_ \Omega \Big[\mu\nabla u: \nabla v  +   ( \lambda + \mu ) \operatorname{div}u \operatorname{div}v\Big]\, {\rm d}x , \qquad e(B,   C) =  -\nu \int_ \Omega  \operatorname{curl}B \cdot \operatorname{curl}   C \,{\rm d} x
\end{equation}
and considering the no slip boundary condition $u|_{ \partial \Omega }=0$ for the velocity. This corresponds in the Lagrangian description to the choice of the subgroup $ \operatorname{Diff}_0( \Omega ) $ of diffeomorphisms fixing the boundary pointwise.
The viscous and resistive barotropic MHD equations with Lagrangian $\ell(u, \rho  , B)$ can be written as follows: seek $u$, $ \rho  $, $B$ with $u|_{ \partial \Omega }=0$ and $B \cdot n|_{ \partial \Omega }=0$ such that
\begin{align}
\left\langle \partial_t \frac{\delta \ell}{\delta u} , v \right\rangle + a\left(\frac{\delta \ell}{\delta u}, u, v\right)  + b\left( \frac{\delta \ell}{\delta \rho},\rho,v \right) + c\left(\frac{\delta \ell}{\delta B},B,v\right)&= d(u,v) &  &\forall\,v, \;\; \text{$v|_{ \partial \Omega }=0$}\label{velocity_res_NS} \\
\langle \partial_t \rho, \sigma \rangle + b(\sigma,\rho,u) &=0, \quad & &\forall\,\sigma \label{density_res_NS}\\
\phantom{\int}\langle \partial_t B, C \rangle + c(C,B,u) &= e(B,C), \quad & &\forall\, C, \;\; C \cdot n|_{ \partial \Omega }=0.\label{magnetic_res_NS}
\end{align}
The boundary condition $ \operatorname{curl}B \times n|_{ \partial \Omega }=0$ emerges from the last equation, while the condition $ \operatorname{div}B(t)=0$ holds if it holds at initial time. For the Lagrangian \eqref{barotropic_ell}, the system \eqref{velocity0}--\eqref{BC} is recovered. While the system \eqref{velocity_res_NS}--\eqref{magnetic_res_NS} is obtained by simply appending the bilinear forms $d$ and $e$ to the Euler-Poincar\'e equations, this system can also be obtained by a variational formulation of Lagrange-d'Alembert type, which extends the Euler-Poincar\'e formulation \eqref{EP_MHD}--\eqref{EP_constraints}, see Appendix~\ref{LdA}.

\paragraph{Balance laws for important quantities.} The balance of total energy $ \mathcal{E}  = \left\langle \frac{\delta \ell}{\delta u}, u \right\rangle -\ell(u, \rho  , B)$ associated to a Lagrangian $\ell$ is found as
\[
\frac{d}{dt} \mathcal{E}  = d(u,u) - e \left( B, \frac{\delta \ell}{\delta B} \right) 
\]
and, in the Euler-Poincar\'e formulation \eqref{velocity_res_NS}--\eqref{magnetic_res_NS}, follows from the property
\[
a(w,u,v)=- a(w,v,u), \quad \forall\; u,v,w
\]
of the trilinear form $a$.

The conservation of total mass $\int_ \Omega \rho \, {\rm d} x$ follows from the property
\[
b(1, \rho  , v)=0, \quad \forall\; \rho  , v.
\]

If $A$ is any vector field satisfying $ \operatorname{curl}A=B$ and $A \times n|_{ \partial \Omega }=0$, the balance of magnetic helicity $\int_ \Omega A \cdot B{\rm d}x$ is found as follows:
\begin{align*} 
\frac{d}{dt} \int_ \Omega A \cdot B{\rm d}x&= \left\langle \partial _t A, B \right\rangle + \left\langle A, \partial _t B \right\rangle \\
&=  \left\langle \partial _t A, \operatorname{curl}   A \right\rangle + \left\langle A,  \partial _t B \right\rangle\\
&=  \left\langle \operatorname{curl}  \partial _t A,   A \right\rangle  + \left\langle A,  \partial _t B \right\rangle\\
&= 2 \left\langle \partial _t B, A \right\rangle \\
&= - 2 c(A,B,u)+ 2e(B,A)\\
&= 2e(B,A),
\end{align*} 
where in the third equality we used $A \times n|_{ \partial \Omega }=0$, in the fifth equality we used \eqref{magnetic_res_NS}, and in the last one we used the following property of $c$:
\[
c(A,B,u)=0 \text{ if }  B= \operatorname{curl}A \text{ and } u|_{ \partial \Omega }=0.
\]
In absence of viscosity, $u|_{ \partial \Omega }=0$ does not hold and one uses
\[
c(A,B,u)=0 \text{ if }  B= \operatorname{curl}A \text{ and } u \cdot n|_{ \partial \Omega }= B \cdot n|_{ \partial \Omega }=0.
\]

\section{Spatial variational discretization} \label{sec:spatial}

We will now construct a spatial discretization of~(\ref{velocity0}-\ref{IC}) using finite elements.  
We make use of the following function spaces:
\begin{align*}
H^1_0(\Omega) &= \{f \in L^2(\Omega) \mid \nabla f \in L^2(\Omega)^d, \, f=0 \text{ on } \partial\Omega \}, \\
H_0(\curl,\Omega) &= 
\begin{cases}
\{ u \in L^2(\Omega)^2 \mid \partial_x u_y - \partial_y u_x \in L^2(\Omega), \, u_x n_y - u_y n_x = 0 \text{ on } \partial\Omega \}, &\mbox{ if } d = 2, \\
\{ u \in L^2(\Omega)^3 \mid \curl u \in L^2(\Omega)^3, \, u \times n = 0 \text{ on } \partial\Omega \}, &\mbox{ if } d = 3, \\
\end{cases} \\
H_0(\dv,\Omega) &= \{u \in L^2(\Omega)^d \mid \dv u \in L^2(\Omega), \, u \cdot n = 0 \text{ on } \partial\Omega \}. 
\end{align*}

Let $\mathcal{T}_h$ be a triangulation of $\Omega$.  We regard $\mathcal{T}_h$ as a member of a family of triangulations parametrized by $h = \max_{K \in \mathcal{T}_h} h_K$, where $h_K = \operatorname{diam}K$ denotes the diameter of a simplex $K$.  We assume that this family is shape-regular, meaning that the ratio $\max_{K \in \mathcal{T}_h} h_K/\rho_K$ is bounded above by a positive constant for all $h>0$.  Here, $\rho_K$ denotes the inradius of $K$.

When $r \ge 0$ is an integer and $K$ is a simplex, we write $P_r(K)$ to denote the space of polynomials on $K$ of degree at most $r$.

Let $r,s \ge 0$ be fixed integers.  To discretize the velocity $u$, we use the continuous Galerkin space
\[
U_h^{\grad} = CG_{r+1}(\mathcal{T}_h)^d := \{u \in H^1_0(\Omega)^d \mid \left. u \right|_K \in P_{r+1}(K)^d, \, \forall K \in \mathcal{T}_h\}.
\]
To discretize the magnetic field $B$, we use the Raviart-Thomas space 
\[
U_h^{\dv} = RT_r(\mathcal{T}_h) :=  \{u \in H_0(\dv,\Omega) \mid \left. u \right|_K \in P_r(K)^d + xP_r(K), \, \forall K \in \mathcal{T}_h\}.
\]
To discretize the density $\rho$, we use the discontinuous Galerkin space
\[
F_h = DG_s(\mathcal{T}_h) := \{ f \in L^2(\Omega) \mid \left. f \right|_K \in P_s(K)^d, \, \forall K \in \mathcal{T}_h\}.
\]
Our method will also make use of an auxiliary space, the Nedelec finite element space of the first kind,
\begin{align*}
U_h^{\curl} &\!=  NED_r(\mathcal{T}_h) := \begin{cases}
\{u \in H_0(\curl,\Omega) \mid \left. u \right|_K \in P_r(K)^2 + (x_2,-x_1) P_r(K), \, \forall K \in \mathcal{T}_h\},\!\!&\!\!\mbox{ if } d=2,\\
\{u \in H_0(\curl,\Omega) \mid \left. u \right|_K \in P_r(K)^3 + x \times P_r(K)^3, \, \forall K \in \mathcal{T}_h\}, \!\!&\!\!\mbox{ if } d=3,
\end{cases}
\end{align*}
which satisfies $\curl U_h^{\curl} \subset U_h^{\dv}$.

We will need consistent discretizations of the trilinear forms $a,b,c$ and the bilinear forms $d,e$.  To construct these, we introduce some notation.  Let $\mathcal{E}_h$ denote the set of interior $(d-1)$-dimensional faces in $\mathcal{T}_h$. On a face $e = K_1 \cap K_2 \in \mathcal{E}_h$, we denote the jump and average of a piecewise smooth scalar function $f$ by
\[
\llbracket f \rrbracket = f_1 n_1 + f_2 n_2, \quad \{f\} = \frac{f_1+f_2}{2},
\] 
where $f_i = \left. f \right|_{K_i}$, $n_1$ is the normal vector to $e$ pointing from $K_1$ to $K_2$, and similarly for $n_2$.  We let $\pi_h^{\grad} : L^2(\Omega)^3 \rightarrow U_h^{\grad}$, $\pi_h^{\curl} : L^2(\Omega)^3 \rightarrow U_h^{\curl}$, $\pi_h^{\dv} : L^2(\Omega)^3 \rightarrow U_h^{\dv}$, and $\pi_h : L^2(\Omega) \rightarrow F_h$ denote the $L^2$-orthogonal projectors onto $U_h^{\grad}$, $U_h^{\curl}$, $U_h^{\dv}$, and $F_h$, respectively.  We define $\curl_h : U_h^{\dv} \rightarrow U_h^{\curl}$ by
\[
\langle \curl_h u, v \rangle = \langle u, \curl v \rangle, \quad \forall v \in U_h^{\curl}.
\]

We define trilinear forms $b_h : F_h \times F_h \times U_h^{\dv} \rightarrow \mathbb{R}$, $c_h : L^2(\Omega)^d \times U_h^{\dv} \times U_h^{\grad} \rightarrow \mathbb{R}$, and a bilinear form $e_h : U_h^{\dv} \times U_h^{\dv} \rightarrow \mathbb{R}$ by
\begin{align*}
b_h(f,g,u) &= -\sum_{K \in \mathcal{T}_h} \int_K (u \cdot \nabla f) g \, {\rm d}x + \sum_{e \in \mathcal{E}_h} \int_e u \cdot \llbracket f \rrbracket \{g\} \, {\rm d}s, \\
c_h(C,B,v) &= \langle C, \curl \pi_h^{\curl} (\pi_h^{\curl}B \times \pi_h^{\curl}v) \rangle, \\
e_h(B,C) &= -\nu \langle \curl_h B, \curl_h C \rangle.\phantom{\int}
\end{align*}
Our choice of $c_h$ is motivated in part by the following lemma.
\begin{lemma} \label{lemma:ch_curl}
The trilinear form $c_h$ satisfies
\begin{equation}
c_h(w,u,v) = 0\; \text{ if } \; \curl w = u.
\end{equation}
\end{lemma}
\begin{proof}
If $\curl w = u$, then we can integrate $c_h$ by parts and use the fact that $\left.n \times \pi_h^{\curl} (\pi_h^{\curl}u \times \pi_h^{\curl}v)\right|_{\partial\Omega} = 0$ to obtain
\begin{align*}
c_h(w,u,v)
&= \langle w, \curl \pi_h^{\curl} (\pi_h^{\curl}u \times \pi_h^{\curl}v) \rangle & \\
&= \langle \curl w,  \pi_h^{\curl} (\pi_h^{\curl}u \times \pi_h^{\curl}v) \rangle &  \\
&= \langle u,  \pi_h^{\curl} (\pi_h^{\curl}u \times \pi_h^{\curl}v) \rangle & \\
&= \langle \pi_h^{\curl} u,  \pi_h^{\curl}u \times \pi_h^{\curl}v \rangle & \\
&= 0. &
\end{align*}
\end{proof}

\medskip 

In the spatially discrete, temporally continuous setting, our method seeks $u : [0,T] \rightarrow U_h^{\grad}$, $\rho : [0,T] \rightarrow F_h$, and $B : [0,T] \rightarrow U_h^{\dv}$ such that
\begin{align}
\langle \sigma, \partial_t \rho \rangle &= -b_h(\sigma,\rho,u), && \forall \sigma \in F_h, \label{rhodoth} \\
\langle C, \partial_t B \rangle &= -c_h(C,B,u), && \forall C \in U_h^{\dv}, \label{Bdoth}
\end{align}
and
\[
\delta \int_0^T \ell(u,\rho,B) \, {\rm d}t = 0
\]
for all variations $\delta u : [0,T] \rightarrow U_h^{\grad}$, $\delta \rho : [0,T] \rightarrow F_h$, and $\delta B : [0,T] \rightarrow U_h^{\dv}$ satisfying
\begin{align}
\langle w, \delta u \rangle &= \langle w, \partial_t v \rangle - a(w,u,v), && \forall w \in U_h^{\grad}, \label{deltauh} \\
\langle \sigma, \delta \rho \rangle &= -b_h(\sigma,\rho,u), && \forall \sigma \in F_h, \label{deltarhoh} \\
\langle C, \delta B \rangle &= -c_h(C,B,u), && \forall C \in U_h^{\dv}, \label{deltaBh}
\end{align}
where $v : [0,T] \rightarrow U_h^{\grad}$ is an arbitrary vector field satisfying $v(0)=v(T)=0$.

Note that~(\ref{rhodoth}-\ref{Bdoth}) and~(\ref{deltauh}-\ref{deltaBh}) are discrete counterparts of the advection laws
\[
\partial_t \rho = -\dv(\rho u), \quad \partial_t B = \curl (u \times B)
\]
and the constraints
\[
\delta u = \partial_t v + [u,v], \quad \delta \rho = -\dv(\rho v), \quad \delta B = \curl(v \times B)
\]
on the variations.

As shown in \cite{GaGB2020}, in the absence of $B$ this variational principle follows from the Hamilton principle on a discrete diffeomorphism group $G_h \subset GL(F_h)$ by applying Euler-Poincar\'e reduction. In particular, the discrete version of $a$ emerging from the Euler-Poincar\'e variational formulation in \cite{GaGB2020} coincides with $a$ on the finite element space $U_h^{\grad}$ used here for the velocity.

The variational principle above yields the following equations for $u \in U_h^{\grad}$, $\rho \in F_h$, $B \in U_h^{\dv}$:
\begin{align}
\left\langle \partial_t \frac{\delta \ell}{\delta u}, v \right\rangle + a\left(\pi_h^{\grad} \frac{\delta \ell}{\delta u}, u, v\right)  + b_h\left( \pi_h \frac{\delta \ell}{\delta \rho},\rho,v \right) + c_h\left(\pi_h^{\dv} \frac{\delta \ell}{\delta B},B,v\right)&= 0, && \forall v \in U_h^{\grad}, \label{velocity_nores_NS_h} \\
\langle \partial_t \rho, \sigma \rangle + b_h(\sigma,\rho,u) &=0, && \forall \sigma \in F_h,\label{density_nores_NS_h} \\
\langle \partial_t B, C \rangle + c_h(C,B,u) &= 0, && \forall C \in U_h^{\dv}. \label{magnetic_nores_NS_h}
\end{align}
We introduce viscosity and resistivity by adding $d(u,v)$ and $e_h(B,C)$ to the right-hand sides of~(\ref{velocity_nores_NS_h}) and~(\ref{magnetic_nores_NS_h}).  The resulting equations read
\begin{align}
\hspace{-0.5cm}\left\langle \partial_t \frac{\delta \ell}{\delta u}, v \right\rangle + a\left(\pi_h^{\grad} \frac{\delta \ell}{\delta u}, u, v\right)  + b_h\left( \pi_h \frac{\delta \ell}{\delta \rho},\rho,v \right) + c_h\left(\pi_h^{\dv} \frac{\delta \ell}{\delta B},B,v\right)&= d(u,v), \!\!\!&\!\!&\!\!\! \forall v \in U_h^{\grad}\!\!, \label{velocity_res_NS_h} \\
\langle \partial_t \rho, \sigma \rangle + b_h(\sigma,\rho,u) &=0, &\!\!&\!\!\!\forall \sigma \in F_h,\label{density_res_NS_h} \\
\langle \partial_t B, C \rangle + c_h(C,B,u) &= e_h(B,C), &\!\!& \!\!\!\forall C \in U_h^{\dv}. \label{magnetic_res_NS_h}
\end{align}
These equations are not implementable in their present form, since the terms involving $c_h$ and $e_h$ contain projections of the test functions $v$ and $C$.  To handle these terms, we use the following lemma.


\begin{lemma} \label{lemma:rectify}
Let $u,B \in U_h^{\dv}$ be arbitrary, and let $J,H,U,E,\alpha,j \in U_h^{\curl}$ be defined by the relations
\begin{align}
\langle J, K \rangle &= -\left\langle \frac{\delta \ell}{\delta B}, \curl K \right\rangle, && \forall K \in U_h^{\curl}, \\
\langle H, G \rangle &= \langle B, G \rangle, && \forall G \in U_h^{\curl}, \\
\langle U, V \rangle &= \langle u, V \rangle, && \forall V \in U_h^{\curl}, \\
\langle E, F \rangle &= -\langle U \times H, F \rangle, && \forall F \in U_h^{\curl}, \label{Edef} \\
\langle \alpha, \beta \rangle &= -\langle J \times H, \beta \rangle, && \forall \beta \in U_h^{\curl}, \\
\langle j, k \rangle &= \left\langle B, \curl k \right\rangle, && \forall k \in U_h^{\curl}.
\end{align}
Then, for every $C \in U_h^{\dv}$ and every $v \in U_h^{\grad}$, we have
\begin{align}
c_h(C,B,u) &= \langle \curl E, C \rangle, \\
c_h\left(\pi_h^{\dv} \frac{\delta \ell}{\delta B},B,v\right) &= \langle \alpha, v \rangle, \label{ch_identity2} \\
e_h(B,C) &= -\nu\langle \curl j, C \rangle. \label{eh_identity}
\end{align}
\end{lemma}
\begin{proof}
We have $H=\pi_h^{\curl} B$ and $U = \pi_h^{\curl} u$ by definition.   Thus,~(\ref{Edef}) implies that
\[
E = -\pi_h^{\curl}(U \times H) = -\pi_h^{\curl} (\pi_h^{\curl} u \times \pi_h^{\curl} B).
\]
It follows that
\[
\langle \curl E, C \rangle = -\langle \curl \pi_h^{\curl} (\pi_h^{\curl} u \times \pi_h^{\curl} B), C \rangle = c_h(C,B,u). 
\] 
To prove~(\ref{ch_identity2}), we use the fact that $\curl U_h^{\curl} \subset U_h^{\dv}$ to write
\begin{align*}
\langle \alpha, v \rangle
&= \langle \alpha, \pi_h^{\curl} v \rangle \\
&= -\langle J \times \pi_h^{\curl} B, \pi_h^{\curl} v \rangle \\
&= -\langle J, \pi_h^{\curl} B \times \pi_h^{\curl} v \rangle \\
&= -\langle J, \pi_h^{\curl} (\pi_h^{\curl} B \times \pi_h^{\curl} v) \rangle \\
&= \left\langle \frac{\delta\ell}{\delta B}, \curl  \pi_h^{\curl} (\pi_h^{\curl}B \times \pi_h^{\curl}v) \right\rangle \\
&= \left\langle \pi_h^{\dv} \frac{\delta\ell}{\delta B}, \curl \pi_h^{\curl} (\pi_h^{\curl}B \times \pi_h^{\curl}v) \right\rangle \\
&= c_h\left( \pi_h^{\dv} \frac{\delta\ell}{\delta B},B,v\right).
\end{align*}
Finally,~(\ref{eh_identity}) follows from the fact that $j = \curl_h B$ by definition, so
\[
-\nu\langle \curl j, C \rangle = -\nu \langle j, \curl_h C \rangle  = e_h(B,C).
\]
\end{proof} 

\medskip 

The preceding lemma shows that~(\ref{velocity_res_NS_h}-\ref{magnetic_res_NS_h}) can be rewritten in the following equivalent way.  We seek $u,w \in U_h^{\grad}$, $B \in U_h^{\dv}$, $\rho,\theta \in F_h$, and $J,H,U,E,\alpha,j \in U_h^{\curl}$ such that
\begin{align}
\left\langle \partial_t \frac{\delta \ell}{\delta u}, v \right\rangle + a\left(w, u, v\right) + b_h(\theta,\rho,v) + \langle \alpha, v \rangle &= d(u,v), && \forall v \in U_h^{\grad}, \label{velocityh_comp} \\
\langle \partial_t \rho, \sigma \rangle + b_h(\sigma,\rho,u) &= 0, && \forall \sigma \in F_h, \label{densityh_comp} \\
\langle \partial_t B, C \rangle + \langle \curl E, C \rangle &= -\nu \langle \curl j, C \rangle, && \forall C \in U_h^{\dv}, \label{magnetich_comp} \\
\langle w, z \rangle &= \left\langle \frac{\delta \ell}{\delta u}, z \right\rangle, && \forall z \in U_h^{\grad}, \label{wh_comp} \\
\langle \theta, \tau \rangle &= \left\langle \frac{\delta \ell}{\delta \rho}, \tau \right\rangle, && \forall \tau \in F_h, \label{thetah_comp} \\
\langle J, K \rangle &= -\left\langle \frac{\delta \ell}{\delta B}, \curl K \right\rangle, && \forall K \in U_h^{\curl}, \\
\langle H, G \rangle &= \langle B, G \rangle, && \forall G \in U_h^{\curl}, \label{Bproj_comp} \\
\langle U, V \rangle &= \langle u, V \rangle, && \forall V \in U_h^{\curl}, \label{uproj_comp}  \\
\langle E, F \rangle &= -\langle U \times H, F \rangle, && \forall F \in U_h^{\curl}, \label{Eh_comp} \\
\langle \alpha, \beta \rangle &= -\langle J \times H, \beta \rangle, && \forall \beta \in U_h^{\curl}, \label{alphah_comp} \\
\langle j, k \rangle &= \left\langle B, \curl k \right\rangle, && \forall k \in U_h^{\curl}. \label{jh_comp}
\end{align}

\begin{remark} \label{remark:jJ}
For Lagrangians that satisfy $\frac{\delta \ell}{\delta B} = -B$, we have $j=J$, so~(\ref{jh_comp}) can be omitted.
\end{remark}

\begin{remark}
The above discretization has several commonalities with the one proposed in~\cite{HuXu2019} for a stationary MHD problem.  The finite element spaces we use for $u$ and $B$ match the ones used there, and our discretization of the term $-\nu \langle \curl B, \curl C\rangle$ matches the one used in Equation 4.3(b-c) of~\cite{HuXu2019}.
\end{remark}

\begin{remark} \label{remark:dingma} The above discretization also has a few commonalities with one that appears in~\cite{DiMa2020}, where a stable finite element method for compressible MHD is proposed and proved to be convergent.  There, the space $CG_2(\mathcal{T}_h)^d$ is used for $u$ and $DG_0(\mathcal{T}_h)$ is used for $\rho$.  These choices coincide with ours when $r=1$ and $s=0$.  However, the authors of~\cite{DiMa2020} use $NED_0(\mathcal{T}_h)$ rather than $RT_r(\mathcal{T}_h)$ for $B$ and treat the boundary condition $\left. B \times n \right|_{\partial\Omega}=0$ rather than $\left. B \cdot n\right|_{\partial\Omega} = \left. \curl B \times n \right|_{\partial\Omega} = 0$.
\end{remark}

\begin{proposition} \label{prop:invariantsh}
If $B(0)$ is exactly divergence-free, then the solution to~(\ref{velocityh_comp}-\ref{jh_comp}) satisfies
\begin{align}
\dv B(t) &\equiv 0, \\
\frac{d}{dt} \int_\Omega \rho \, {\rm d}x &= 0, \\
\frac{d}{dt} \mathcal{E} &= d(u,u) - e_h\left(B, \pi_h^{\dv} \frac{\delta \ell}{\delta B} \right), \label{energyh} \\
\frac{d}{dt} \int_\Omega A \cdot B \, {\rm d}x &= 2e_h(B,\pi_h^{\dv} A), \label{magnetichelicityh}
\end{align}
for all $t$.  Here, $\mathcal{E} = \langle \frac{\delta\ell}{\delta u}, u \rangle - \ell(u,\rho,B)$ denotes the total energy of the system, and $A$ denotes any vector field satisfying $\nabla \times A = B$ and $\left.A \times n \right|_{\partial\Omega} = 0$.
\end{proposition}
\begin{proof}
Since $\curl U_h^{\curl} \subset U_h^{\dv}$, the magnetic field equation~(\ref{magnetich_comp}) implies that the relation
\[
\partial_t B + \curl E = -\nu \curl j
\]
holds pointwise.  Taking the divergence of both sides shows that $\partial_t \dv B = 0$, so $B(t)$ is divergence-free for all $t$.

Taking $\sigma=1$ in the density equation~(\ref{densityh_comp}) shows that
\[
\frac{d}{dt} \int_\Omega \rho \, {\rm d}x  = \langle \partial_t \rho, 1 \rangle = -b_h(1,\rho,u) = 0.
\]

To compute the rate of change of the energy, we take $v=u$ in~(\ref{velocity_res_NS_h}), $\sigma=-\pi_h \frac{\delta\ell}{\delta \rho}$ in~(\ref{density_res_NS_h}), and $C=-\pi_h^{\dv} \frac{\delta\ell}{\delta B}$ in~(\ref{magnetic_res_NS_h}).  Adding the three equations yields
\begin{equation*}
\left\langle \partial_t \frac{\delta \ell}{\delta u}, u \right\rangle - \left\langle \partial_t \rho, \pi_h \frac{\delta\ell}{\delta \rho} \right\rangle - \left\langle \partial_t B, \pi_h^{\dv} \frac{\delta \ell}{\delta B} \right\rangle = d(u,u) - e_h\left(B, \pi_h^{\dv} \frac{\delta \ell}{\delta B} \right).
\end{equation*}
Since $\partial_t \rho \in F_h$ and $\partial_t B \in U_h^{\dv}$, this simplifies to 
\begin{equation*}
\left\langle \partial_t \frac{\delta \ell}{\delta u}, u \right\rangle - \left\langle \partial_t \rho,  \frac{\delta\ell}{\delta \rho} \right\rangle - \left\langle \partial_t B, \frac{\delta \ell}{\delta B} \right\rangle = d(u,u) - e_h\left(B, \pi_h^{\dv} \frac{\delta \ell}{\delta B} \right),
\end{equation*}
which is equivalent to
\[
\frac{d}{dt} \left( \left\langle \frac{\delta \ell}{\delta u}, u \right\rangle - \ell(u,\rho,B) \right) = d(u,u) - e_h\left(B, \pi_h^{\dv} \frac{\delta \ell}{\delta B} \right).
\]

For the magnetic helicity, we compute
\begin{align*} 
\frac{d}{dt} \int_ \Omega A \cdot B \, {\rm d}x &= \left\langle \partial _t A, B \right\rangle + \left\langle A, \partial _t B \right\rangle \\
&=  \left\langle \partial _t A, \curl  A \right\rangle + \left\langle A,  \partial _t B \right\rangle\\
&=  \left\langle \curl \partial _t A,   A \right\rangle + \left\langle A,  \partial _t B \right\rangle\\
&= 2 \left\langle \partial _t B, A \right\rangle \\
&= 2 \langle \partial _t B, \pi_h^{\dv} A \rangle \\
&= - 2 c_h(\pi_h^{\dv} A,B,u)+ 2e_h(B,\pi_h^{\dv} A).
\end{align*} 
Since $\curl U_h^{\curl} \subset U_h^{\dv}$, we have
\begin{align*}
c_h(\pi_h^{\dv} A,B,u) 
&= \langle \pi_h^{\dv} A, \curl \pi_h^{\curl}(\pi_h^{\curl}B \times \pi_h^{\curl} u) \rangle \\
&= \langle A, \curl \pi_h^{\curl}(\pi_h^{\curl}B \times \pi_h^{\curl} u) \rangle \\
&= c_h(A,B,u).
\end{align*}
Thus,
\[
\frac{d}{dt} \int_ \Omega A \cdot B \, {\rm d}x = - 2 c_h(A,B,u)+ 2e_h(B, \pi_h^{\dv} A).
\]
The first term vanishes by Lemma~\ref{lemma:ch_curl}, yielding~(\ref{magnetichelicityh}).
\end{proof}

\begin{remark}
The proposition above continues to hold if we omit the projection of $\frac{\delta \ell}{\delta u}$ onto $U_h^{\grad}$ in~(\ref{wh_comp}).  We find it advantageous to do this for efficiency.  As an illustration, let us consider the setting where $\ell(u,\rho,B) = \int_\Omega [\frac{1}{2}\rho|u|^2 - \epsilon(\rho) - \frac{1}{2}|B|^2] \, {\rm d}x$.  If we omit~(\ref{wh_comp}) and invoke  Remark~\ref{remark:jJ}, then the method seeks $u \in U_h^{\grad}$, $B \in U_h^{\dv}$, $\rho,\theta \in F_h$, and $J,H,U,E,\alpha \in U_h^{\curl}$ such that
\begin{align}
\left\langle \partial_t (\rho u), v \right\rangle + a\left(\rho u, u, v\right) + b_h(\theta,\rho,v) + \langle \alpha, v \rangle &= d(u,v), && \forall v \in U_h^{\grad}, \label{velocityh_mhd} \\
\langle \partial_t \rho, \sigma \rangle + b_h(\sigma,\rho,u) &= 0, && \forall \sigma \in F_h, \label{densityh_mhd} \\
\langle \partial_t B, C \rangle + \langle \curl E, C \rangle &= -\nu \langle \curl J, C \rangle, && \forall C \in U_h^{\dv}, \label{magnetich_mhd} \\
\langle \theta, \tau \rangle &= \left\langle \frac{1}{2}|u|^2 - \frac{\partial \epsilon}{\partial \rho}, \tau \right\rangle, && \forall \tau \in F_h, \label{thetah_mhd} \\
\langle J, K \rangle &= \left\langle B, \curl K \right\rangle, && \forall K \in U_h^{\curl}, \\
\langle H, G \rangle &= \langle B, G \rangle, && \forall G \in U_h^{\curl}, \label{Bproj_mhd} \\
\langle U, V \rangle &= \langle u, V \rangle, && \forall V \in U_h^{\curl}, \label{uproj_mhd}  \\
\langle E, F \rangle &= -\langle U \times H, F \rangle, && \forall F \in U_h^{\curl}, \label{Eh_mhd} \\
\langle \alpha, \beta \rangle &= -\langle J \times H, \beta \rangle, && \forall \beta \in U_h^{\curl}. \label{alphah_mhd}
\end{align}
In this setting, the energy identity~(\ref{energyh}) becomes
\[
\frac{d}{dt} \int_\Omega \Big[\frac{1}{2} \rho |u|^2 + \epsilon(\rho) + \frac{1}{2}  |B|^2 \Big]\, {\rm d}x = d(u,u) + e_h(B,B).
\]
\end{remark}

\medskip 

\section{Temporal discretization}

In this section, we design a temporal discretization of~(\ref{velocityh_mhd}-\ref{alphah_mhd}) for Lagrangians of the form
\[
\ell(u,\rho,B) = \int_\Omega \Big[\frac{1}{2}\rho|u|^2 - \epsilon(\rho) - \frac{1}{2}|B|^2\Big] \, {\rm d}x.
\]
Our temporal discretization will retain all of the structure-preserving properties of our spatial discretization: energy balance, magnetic helicity balance, total mass conservation, and $\dv B = 0$.

We adopt the following notation.  For a fixed time step $\Delta t > 0$, we denote $t_k = k\Delta t$.  The value of the approximate solution $u \in U_h^{\grad}$ at time $t_k$ is denoted $u_k$, and likewise for $\rho$ and $B$.  The auxiliary variables $\theta \in F_h$ and $J,H,U,E,\alpha \in U_h^{\curl}$ will play a role in our calculations, but we do not index them with a subscript $k$.  We write $u_{k+1/2} = \frac{u_k+u_{k+1}}{2}$, $\rho_{k+1/2} = \frac{\rho_k+\rho_{k+1}}{2}$, $B_{k+1/2} = \frac{B_k+B_{k+1}}{2}$, and $(\rho u)_{k+1/2} =  \frac{\rho_k u_k+\rho_{k+1} u_{k+1}}{2}$.      We will also make use of the bivariate function
\[
\delta(x,y) = \frac{\epsilon(y)-\epsilon(x)}{y-x}.
\]

Given $u_k, \rho_k, B_k$, our method steps from time $t_k$ to $t_{k+1}$ by solving
\begin{align}
\left\langle \frac{\rho_{k+1} u_{k+1} - \rho_k u_k}{\Delta t}, v \right\rangle + a\left((\rho u)_{k+1/2}, u_{k+1/2}, v\right) &&\nonumber\\+\, b_h(\theta,\rho_{k+1/2},v) + \langle \alpha, v \rangle &= d(u_{k+1/2},v), && \forall v \in U_h^{\grad}, \label{velocityh_dt} \\
\left\langle \frac{\rho_{k+1}-\rho_k}{\Delta t}, \sigma \right\rangle + b_h(\sigma,\rho_{k+1/2},u_{k+1/2}) &= 0, && \forall \sigma \in F_h, \label{densityh_dt} \\
\left\langle \frac{B_{k+1}-B_k}{\Delta t}, C \right\rangle + \langle \curl E, C \rangle &= -\nu \langle \curl J, C \rangle, && \forall C \in U_h^{\dv}, \label{magnetich_dt}
\end{align}
for $u_{k+1},\rho_{k+1},B_{k+1}$.  Here, $\theta$, $\alpha$, $E$, and $J$ (as well as $H$ and $U$) are defined by
\begin{align}
\langle \theta, \tau \rangle &= \left\langle \frac{1}{2}u_k \cdot u_{k+1} - \delta(\rho_k,\rho_{k+1}), \tau \right\rangle, && \forall \tau \in F_h, \label{thetah_dt} \\
\langle J, K \rangle &= \left\langle B_{k+1/2}, \curl K \right\rangle, && \forall K \in U_h^{\curl}, \label{Jh_dt} \\
\langle H, G \rangle &= \langle B_{k+1/2}, G \rangle, && \forall G \in U_h^{\curl}, \label{Bproj_dt} \\
\langle U, V \rangle &= \langle u_{k+1/2}, V \rangle, && \forall V \in U_h^{\curl}, \label{uproj_dt}  \\
\langle E, F \rangle &= -\langle U \times H, F \rangle, && \forall F \in U_h^{\curl}, \label{Eh_dt} \\
\langle \alpha, \beta \rangle &= -\langle J \times H, \beta \rangle, && \forall \beta \in U_h^{\curl}. \label{alphah_dt}
\end{align}

Notice that we used the midpoint rule everywhere above except in the definition of $\theta$, where we used 
\[
\frac{1}{2}u_k \cdot u_{k+1} - \delta(\rho_k,\rho_{k+1})
\]
instead of  
\[
\frac{1}{2} |u_{k+1/2}|^2 - \left.\frac{\partial\epsilon}{\partial\rho}\right|_{\rho=\rho_{k+1/2}}
\]
to discretize $\frac{1}{2} |u|^2 -\frac{\partial\epsilon}{\partial\rho}$.  This will allow us to take advantage of the identity
\begin{equation} \label{rhousquared}
\begin{split}
&\frac{1}{\Delta t} \int_\Omega \Big[ \frac{1}{2}\rho_{k+1} |u_{k+1}|^2 + \epsilon(\rho_{k+1}) - \frac{1}{2} \rho_k |u_k|^2 - \epsilon(\rho_k)\Big]  \, {\rm d}x \\
&= \left\langle \frac{\rho_{k+1} u_{k+1} - \rho_k u_k}{ \Delta t }, \frac{u_k+u_{k+1}}{2} \right\rangle - \left\langle \frac{\rho_{k+1}-\rho_k}{\Delta t}, \frac{1}{2} u_k \cdot u_{k+1} - \delta(\rho_k,\rho_{k+1}) \right\rangle
\end{split}
\end{equation}
when we prove energy conservation below.

\begin{proposition} \label{prop:invariantsh_dt}
If $B_0$ is exactly divergence-free, then the solution to~(\ref{velocityh_dt}-\ref{alphah_dt}) satisfies
\begin{align}
\dv B_k &\equiv 0, \label{divBh_dt} \\
\int_\Omega \rho_{k+1} \, {\rm d}x &= \int_\Omega \rho_k \, {\rm d}x, \label{massh_dt} \\
\frac{\mathcal{E}_{k+1}-\mathcal{E}_k}{\Delta t} &= d(u_{k+1/2},u_{k+1/2}) + e_h(B_{k+1/2}, B_{k+1/2}), \label{energyh_dt} \\
\frac{1}{\Delta t} \left( \int_\Omega A_{k+1} \cdot B_{k+1} \, {\rm d}x - \int_\Omega A_k \cdot B_k \, {\rm d}x \right) &= 2e_h(B_{k+1/2},\pi_h^{\dv} A_{k+1/2}), \label{magnetichelicityh_dt}
\end{align}
for all $k$.  Here, $\mathcal{E}_k = \langle \frac{\delta\ell}{\delta u_k}, u_k \rangle - \ell(u_k,\rho_k,B_k)$ denotes the total energy of the system, and $A_k$ denotes any vector field satisfying $\nabla \times A_k = B_k$ and $\left.A_k \times n \right|_{\partial\Omega} = 0$.
\end{proposition}
\begin{proof}
The magnetic field equation~(\ref{magnetich_dt}) implies that the relation
\[
\frac{B_{k+1}-B_k}{\Delta t} + \curl E = -\nu \curl J
\]
holds pointwise, so taking the divergence of both sides proves~(\ref{divBh_dt}).  Conservation of total mass (i.e.~(\ref{massh_dt})) is proved by taking $\sigma \equiv 1$ in the density equation~(\ref{densityh_dt}).

To prove~(\ref{energyh_dt}-\ref{magnetichelicityh_dt}), we introduce some notation.  Let $D_{\Delta t} (\rho u) = \frac{\rho_{k+1} u_{k+1} - \rho_k u_k}{\Delta t}$, $D_{\Delta t} B = \frac{B_{k+1}-B_k}{\Delta t}$, etc.  To reduce notational clutter, we will suppress the subscript $k+1/2$ on quantities evaluated at $t_{k+1/2}$.  Thus, we abbreviate $u_{k+1/2}$, $B_{k+1/2}$, $\rho_{k+1/2}$, and $(\rho u)_{k+1/2}$ as $u$, $B$, $\rho$, and $\rho u$, respectively.   Using Lemma~\ref{lemma:rectify}, equations~(\ref{velocityh_dt}-\ref{alphah_dt}) can be rewritten in the form
\begin{align}
\left\langle D_{\Delta t} (\rho u), v \right\rangle + a\left(\rho u, u, v\right) + b_h\left( \theta,\rho,v \right) - c_h\left(B,B,v\right)&= d(u,v), && \forall v \in U_h^{\grad}, \label{velocity_res_NS_h_dt} \\
\langle D_{\Delta t} \rho, \sigma \rangle + b_h(\sigma,\rho,u) &=0, && \forall \sigma \in F_h,\label{density_res_NS_h_dt} \\
\langle D_{\Delta t} B, C \rangle + c_h(C,B,u) &= e_h(B,C), && \forall C \in U_h^{\dv}, \label{magnetic_res_NS_h_dt}
\end{align}
where
\[
\theta = \pi_h \left( \frac{1}{2}u_k \cdot u_{k+1} - \delta(\rho_k,\rho_{k+1}) \right).
\]
Taking $v=u$, $\sigma=-\theta$, and $C=B$ in~(\ref{velocity_res_NS_h_dt}-\ref{magnetic_res_NS_h_dt}) and adding the three equations gives
\[
\langle D_{\Delta t} (\rho u), u \rangle - \langle D_{\Delta t} \rho, \theta \rangle + \langle D_{\Delta t} B, B \rangle = d(u,u) + e_h(B,B).
\]
Written in full detail, this reads
\[\begin{split}
\left\langle \frac{\rho_{k+1} u_{k+1} - \rho_k u_k}{ \Delta t }, \frac{u_k+u_{k+1}}{2} \right\rangle - \left\langle \frac{\rho_{k+1}-\rho_k}{\Delta t}, \pi_h \left( \frac{1}{2} u_k \cdot u_{k+1} - \delta(\rho_k,\rho_{k+1}) \right) \right\rangle \\
+ \left\langle \frac{B_{k+1}-B_k}{\Delta t}, \frac{B_k+B_{k+1}}{2} \right\rangle = d(u_{k+1/2},u_{k+1/2}) + e_h(B_{k+1/2},B_{k+1/2}).
\end{split}\]
Since $\frac{\rho_{k+1}-\rho_k}{\Delta t} \in F_h$, we can remove $\pi_h$ from the second term above and use the identity~(\ref{rhousquared}) to rewrite the equation above as
\[\begin{split}
&\frac{1}{\Delta t} \int_\Omega \left( \frac{1}{2}\rho_{k+1} |u_{k+1}|^2 + \epsilon(\rho_{k+1})  - \frac{1}{2} \rho_k |u_k|^2 - \epsilon(\rho_k) \right)  \, {\rm d}x \\ &+ \frac{1}{\Delta t} \int_\Omega \left( \frac{1}{2}|B_{k+1}|^2 - \frac{1}{2}|B_k|^2 \right) \, {\rm d}x  = d(u_{k+1/2},u_{k+1/2}) + e_h(B_{k+1/2},B_{k+1/2}).
\end{split}\]
This proves~(\ref{energyh_dt}).  

To prove~(\ref{magnetichelicityh_dt}), we revert to our abbreviated notation and compute
\begin{align*}
\frac{1}{\Delta t} \left( \langle A_{k+1}, B_{k+1} \rangle - \langle A_k, B_k \rangle \right) 
&= \langle D_{\Delta t} A, B \rangle + \langle A, D_{\Delta t} B \rangle \\
&= \langle D_{\Delta t} A, \curl A \rangle + \langle A, D_{\Delta t} B \rangle \\
&=  \left\langle \curl D_{\Delta t} A,   A \right\rangle + \left\langle A,  D_{\Delta t} B \right\rangle\\
&= 2 \left\langle D_{\Delta t} B, A \right\rangle \\
&= 2 \left\langle D_{\Delta t} B, \pi_h^{\dv} A \right\rangle \\
&= - 2 c_h(\pi_h^{\dv} A,B,u)+ 2e_h(B,\pi_h^{\dv} A) \\
&= - 2 c_h(A,B,u)+ 2e_h(B,\pi_h^{\dv} A).
\end{align*}
The first term vanishes by Lemma~\ref{lemma:ch_curl}, yielding~(\ref{magnetichelicityh_dt}).
\end{proof}

\subsection{Enhancements and Extensions}

Below we discuss several enhancements and extensions of the numerical method~(\ref{velocityh_dt}-\ref{alphah_dt}).

\paragraph{Two dimensions.}  

The two-dimensional setting can be treated exactly as above, except one must distinguish between vector fields in the plane ($u$, $B$, $H$, $U$, and $\alpha$) and those orthogonal to it ($J$ and $E$).  We thus identify $J$ and $E$ (as well as the test functions $K$ and $F$ appearing in~(\ref{Jh_dt}) and~(\ref{Eh_dt})) with scalar fields, and we use the continuous Galerkin space  
\[
\{u \in H^1_0(\Omega) \mid \left. u \right|_K \in P_{r+1}(K), \, \forall K \in \mathcal{T}_h \}
\]
to discretize them.

\paragraph{Upwinding.}  To help reduce artificial oscillations in discretizations of scalar advection laws like~(\ref{densityh_dt}), it is customary to incorporate upwinding.  As discussed in~\cite{GaGB2021}, this can be accomplished without interfering with any balance laws by introducing a $u$-dependent trilinear form
\[
\widetilde{b}_h(u; f,g,v) = b_h(f,g,v) + \sum_{e \in \mathcal{E}_h} \int_e \beta_e(u) \left( \frac{v \cdot n}{u \cdot n} \right) \llbracket f \rrbracket \cdot \llbracket g \rrbracket \, \mathrm{d}s,
\]
where $\{\beta_e(u)\}_{e \in \mathcal{E}_h}$ are nonnegative scalars.  One then replaces every appearance of $b_h(\cdot,\cdot,\cdot)$ in~(\ref{velocityh_dt}-\ref{densityh_dt}) by $\widetilde{b}_h(u_{k+1/2};\,\cdot,\cdot,\cdot)$.  It is not hard to see that this enhancement has no effect on the balance laws~(\ref{divBh_dt}-\ref{magnetichelicityh_dt}).  That is, Proposition~\ref{prop:invariantsh_dt} continues to hold.  We used this upwinding strategy with \[
\beta_e(u) = \frac{1}{\pi} (u \cdot n) \arctan\left( \frac{u \cdot n}{0.01} \right) \approx \frac{1}{2}|u \cdot n|
\]
in all of the numerical experiments that appear in Section~\ref{sec:numerical}.

\paragraph{Zero viscosity.}

When the fluid viscosity coefficients $\mu$ and $\lambda$ vanish, the boundary condition $\left. u \right|_{\partial\Omega} = 0$ changes to $\left. u \cdot n \right|_{\partial\Omega} = 0$, and the term $d(u,v)$ on the right-hand side of~(\ref{velocity_res_NS}) vanishes.  
To handle this setting, we modify the scheme~(\ref{velocityh_dt}-\ref{alphah_dt}) as follows.  We use the space $U_h^{\dv}$ instead of $U_h^{\grad}$ to discretize $u$ (as well as the test function $v$ appearing on the right-hand side of~(\ref{velocityh_dt})), and we replace the term $a((\rho u)_{k+1/2}, u_{k+1/2}, v)$ by 
\[
a_h(u_{k+1/2}; (\rho u)_{k+1/2}, u_{k+1/2}, v),
\]
where $a_h(U;\cdot,\cdot,\cdot)$ denotes the $U$-dependent trilinear form
\[\begin{split}
a_h(U; w,u,v) = \sum_{K \in \mathcal{T}_h} \int_K w \cdot (v \cdot \nabla u - u \cdot \nabla v) \, {\rm d}x + \sum_{e \in \mathcal{E}_h} \int_e n \times (\{w\}+\alpha_e(U)\llbracket w \rrbracket) \cdot \llbracket u \times v \rrbracket \, {\rm d}s.
\end{split}\]
Here, $\llbracket w\rrbracket$ and $\{w\}$ denote the jump and average, respectively, of $w$ across the edge $e$, and $\{\alpha_e(U)\}_{e \in \mathcal{E}_h}$ are nonnegative scalars. We took $\alpha_e(U)=\beta_e(U)/(U \cdot n)$, which is a way of incorporating upwinding in the momentum advection; see~\cite{GaGB2020}.  Note that $a_h(U;w,u,v)$ reduces to $a(w,u,v)$ when $u,v \in U_h^{\grad}$ since the term $\llbracket u \times v \rrbracket$ vanishes. Here also, Proposition~\ref{prop:invariantsh_dt} continues to hold (with $b_h$ or $\widetilde{b}_h$), where we now have  $d=0$ in \eqref{energyh_dt}.  If additionally $\nu=0$, i.e. resistivity is absent, then we have $e_h=0$ in~(\ref{energyh_dt}-\ref{magnetichelicityh_dt}) as well.

\paragraph{Entropy and gravitational forces.}

The extension of our scheme to full compressible ideal MHD subject to gravitational and Coriolis forces is straightforward.  Here we describe the incorporation of the entropy density $s$ and the gravitational potential $\phi$, omitting Coriolis forces for simplicity.  The Lagrangian~(\ref{baroclinic_ell}) is thus
\[
\ell(u, \rho ,s , B)= \int_ \Omega  \Big[\frac{1}{2} \rho  |u|^2 - \epsilon ( \rho ,s ) - \rho  \phi - \frac{1}{2} |B| ^2\Big] {\rm d} x,
\]
and $s$ is treated as an advected parameter: $\partial_t s + \dv (su) = 0$.

In the discrete setting, this leads to the following modifications of our basic scheme.   We introduce an additional unknown $s_k \in F_h$, the discrete entropy density, which is advected according to
\begin{align}
\left\langle \frac{s_{k+1}-s_k}{\Delta t}, \sigma \right\rangle + \widetilde{b}_h(u_{k+1/2};\sigma,s_{k+1/2},u_{k+1/2}) &= 0, && \forall \sigma \in F_h. \label{entropyh_dt_baro}
\end{align}
In place of~(\ref{thetah_dt}), we define two auxiliary variables $\theta_1,\theta_2 \in F_h$ by
\begin{align}
\langle \theta_1, \tau \rangle &= \left\langle \frac{1}{2}u_k \cdot u_{k+1} - \phi - \frac{\delta_1(\rho_k,\rho_{k+1},s_k) + \delta_1(\rho_k,\rho_{k+1},s_{k+1})}{2}, \tau \right\rangle, && \forall \tau \in F_h, \label{thetah_dt_baro1} \\
\langle \theta_2, \tau \rangle &= \left\langle - \frac{\delta_2(s_k,s_{k+1},\rho_k) + \delta_2(s_k,s_{k+1},\rho_{k+1})}{2}, \tau \right\rangle, && \forall \tau \in F_h, \label{thetah_dt_baro2} 
\end{align}
where
\begin{align*}
\delta_1(\rho,\rho',s) &= \frac{\epsilon(\rho',s)-\epsilon(\rho,s)}{\rho'-\rho}, \\
\delta_2(s,s',\rho) &= \frac{\epsilon(\rho,s')-\epsilon(\rho,s)}{s'-s}.
\end{align*}
Then we replace the term
\[
b_h(\theta,\rho_{k+1/2},v)
\]
in~(\ref{velocityh_dt}) by
\[
\widetilde{b}_h(u_{k+1/2},\theta_1,\rho_{k+1/2},v) + \widetilde{b}_h(u_{k+1/2},\theta_2,s_{k+1/2},v).
\]
The resulting scheme satisfies all of the balance laws~(\ref{divBh_dt}-\ref{magnetichelicityh_dt}), this time with
\begin{align*}
\mathcal{E}_k 
&= \left\langle \frac{\delta\ell}{\delta u_k}, u_k \right\rangle - \ell(u_k,\rho_k,B_k,s_k) \\
&= \int_\Omega \Big[ \frac{1}{2}\rho_k |u_k|^2 + \frac{1}{2}|B_k|^2 + \epsilon(\rho_k,s_k) + \rho_k \phi \Big] \, {\rm d}x.
\end{align*}
Note that this scheme is especially relevant in the absence of viscosity and resistivity, i.e. $d(u,v)=e_h(B,C)=0$, since the entropy density $s$ is treated as an advected parameter above.  Nevertheless, we have found it advantageous in some of our numerical experiments to continue to include the terms $d(u,v)$ and $e_h(B,C)$ to promote stability.

\section{Numerical examples} \label{sec:numerical}

To illustrate the structure-preserving properties of our numerical method, we solved the compressible barotropic MHD equations~(\ref{velocity0}-\ref{IC}) with $\mu=\nu=\lambda=0$ and $\epsilon(\rho) = \rho^{5/3}$ on a three-dimensional domain $\Omega = [-1,1]^3$  with initial conditions
\begin{align}
u(x,y,z,0) &= \left( \sin(\pi x)\cos(\pi y)\cos(\pi z), \, \cos(\pi x)\sin(\pi y)\cos(\pi z), \, \cos(\pi x)\cos(\pi y)\sin(\pi z) \right), \label{u03d} \\
B(x,y,z,0) &= \curl \left( (1-x^2)(1-y^2)(1-z^2) v \right), \label{B03d} \\
\rho(x,y,z,0) &= 2 + \sin(\pi x)\sin(\pi y)\sin(\pi z), \label{rho03d}
\end{align}
where $v = \frac{1}{2}(\sin \pi x, \sin \pi y, \sin \pi z)$.  We used a uniform triangulation $\mathcal{T}_h$ of $\Omega$ with maximum element diameter $h = \sqrt{3}/2$, and we used the finite element spaces specified in Section~\ref{sec:spatial} of order $r=s=0$.  We used a time step $\Delta t = 0.005$.   Figure~\ref{fig:invariants3d} shows that the scheme preserves energy, magnetic helicity, mass, and $\dv B = 0$ to machine precision, while cross helicity drifts slightly.

\begin{figure}
\centering
\hspace{-0.5in}
\begin{tikzpicture}
\begin{groupplot}[
group style={
    group name=my plots,
    group size=1 by 1,
    xlabels at=edge bottom,
    ylabels at=edge left,
    horizontal sep=2cm,vertical sep=2cm},
    ymode=log,xlabel=$t$,ylabel=$|F(t)-F(0)|$,
    ylabel style={yshift=-0.1cm},
    ymin = 10^-20,
    ymax = 1,
    legend style={at={(1.0,0.0)},
	anchor=south,
	column sep=1ex,
	text=black},
    legend columns=6
    ]
\nextgroupplot[legend to name=testLegend3d,title={$$}]
\addplot[blue,thick] table [x expr=\coordindex*0.005, y expr=abs(\thisrowno{0})]{Data/3d_invariants.dat};
\addplot[red,thick] table [x expr=\coordindex*0.005, y expr=abs(\thisrowno{1})]{Data/3d_invariants.dat};
\addplot[black,thick] table [x expr=\coordindex*0.005, y expr=abs(\thisrowno{2})]{Data/3d_invariants.dat};
\addplot[blue,dashed,thick] table [x expr=\coordindex*0.005, y expr=abs(\thisrowno{3})]{Data/3d_invariants.dat};
\addplot[red,dashed,thick] table [x expr=\coordindex*0.005, y expr=abs(\thisrowno{4})]{Data/3d_invariants.dat};
\legend{$\int \rho \, {\rm d}x$ \\ $\int[ \frac{1}{2}  \rho |u|^2 + \frac{1}{2}|B|^2 + \epsilon(\rho) ]\, {\rm d}x$ \\$\int u \cdot B \, {\rm d}x$ \\ $\|\dv B\|_{L^2(\Omega)}$ \\ $\int A \cdot B \, {\rm d}x$\\}
\end{groupplot}
\end{tikzpicture}
\ref{testLegend3d}
\caption{Evolution of mass, energy, cross-helicity, $\|\dv B\|_{L^2(\Omega)}$, and magnetic helicity during a simulation in three dimensions.  The absolute deviations $|F(t)-F(0)|$ are plotted for each such quantity $F(t)$.}
\label{fig:invariants3d}
\end{figure}

Next, we simulated a magnetic Rayleigh-Taylor instability on the domain $\Omega = [0,L] \times [0,4L]$ with $L=\frac{1}{4}$.  We chose
\[
\epsilon(\rho,s) = K e^{s/(C_v \rho)} \rho^\gamma
\]
with $C_v=K=1$ and $\gamma=\frac{5}{3}$, and we set $\mu=\nu=\lambda=0.01$ and $\phi=-y$, which corresponds to an upward gravitational force.
As initial conditions, we took
\begin{align*}
\rho(x,y,0) &= 1.5-0.5\tanh\left( \frac{y-0.5}{0.02} \right), \\
u(x,y,0) &= \left( 0, -0.025\sqrt{\frac{\gamma p(x,y)}{\rho(x,y,0)}} \cos(8\pi x) \exp\left( -\frac{(y-0.5)^2}{0.09}\right) \right), \\
s(x,y,0) &= C_v \rho(x,y,0) \log\left( \frac{p(x,y)}{(\gamma-1)K\rho(x,y,0)^\gamma} \right), \\
B(x,y,0) &= (B_0,0),
\end{align*}
where
\[
p(x,y) = 1.5y + 1.25 + (0.25-0.5y) \tanh\left( \frac{y-0.5}{0.02} \right).
\]
This system is known to exhibit instability when $B_0 <B_c= \sqrt{( \rho  _h - \rho  _l) gL}$, where here $ \rho  _h=2$, $ \rho  _l=1$, $g=1$, $L=1/4$, \cite{Ch1961}.

We imposed boundary conditions $u = (B-(B_0,0)) \cdot n = \curl B \times n = 0$ on $\partial\Omega$.  We triangulated $\Omega$ with a uniform triangulation $\mathcal{T}_h$ having maximum element diameter $h=2^{-7}$, and we used the finite element spaces specified in Section~\ref{sec:spatial} of order $r=s=0$.  We ran simulations from $t=0$ to $t=5$ using a time step $\Delta t = 0.005$.  Plots of the computed mass density for various choices of $B_0$ are shown in Figures~\ref{fig:B0p2}-\ref{fig:B0p8}.  The figures indicate that the scheme correctly predicts instability for $B_0<B_c=0.5$ (Figures~\ref{fig:B0p2}-\ref{fig:B0p4}) and stability for $B_0>B_c=0.5$ (Figures~\ref{fig:B0p6}-\ref{fig:B0p8}).  

The evolution of energy, cross-helicity, mass, and $\|\dv B\|_{L^2(\Omega)}$ during these simulations is plotted in Figure~\ref{fig:invariants}.  (Magnetic helicity is not plotted since it is trivially preserved in two dimensions if we take $A$ to be orthogonal to the plane.)  As predicted by Proposition~\ref{prop:invariantsh_dt}, Figure~\ref{fig:invariants} shows that energy decayed monotonically, while mass and $\dv B = 0$ were preserved to machine precision.  Interestingly, the cross-helicity drifted by less than $2.5 \times 10^{-4}$ in these experiments, even though the scheme is not designed to preserve it.

\begin{figure}
\centering
\plotsnapshots{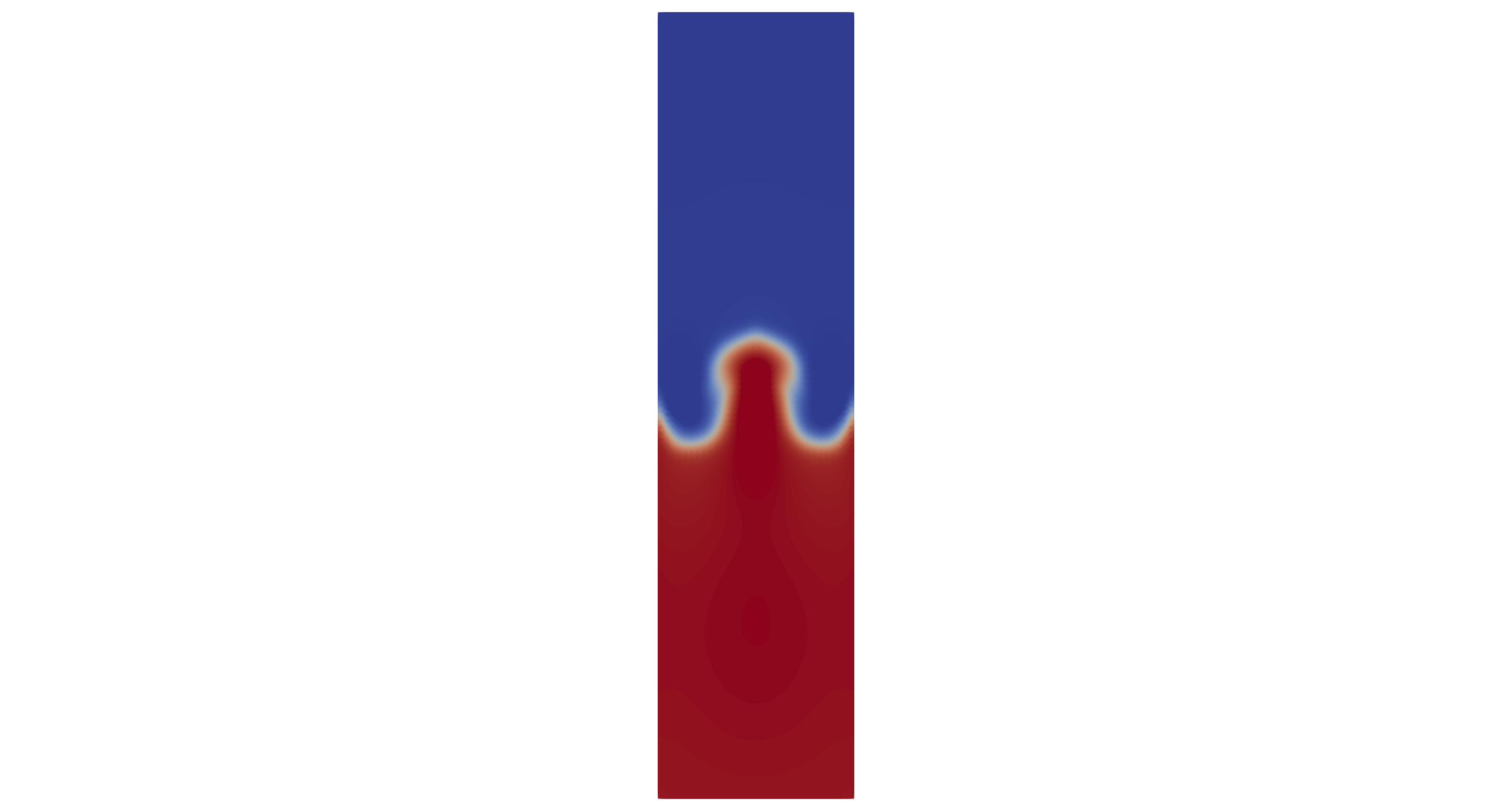}
\caption{Contours of the mass density at $t=0,1,2,3,4,5$ when $B_0=0.2$.}
\label{fig:B0p2}
\end{figure}

\begin{figure}
\centering
\plotsnapshots{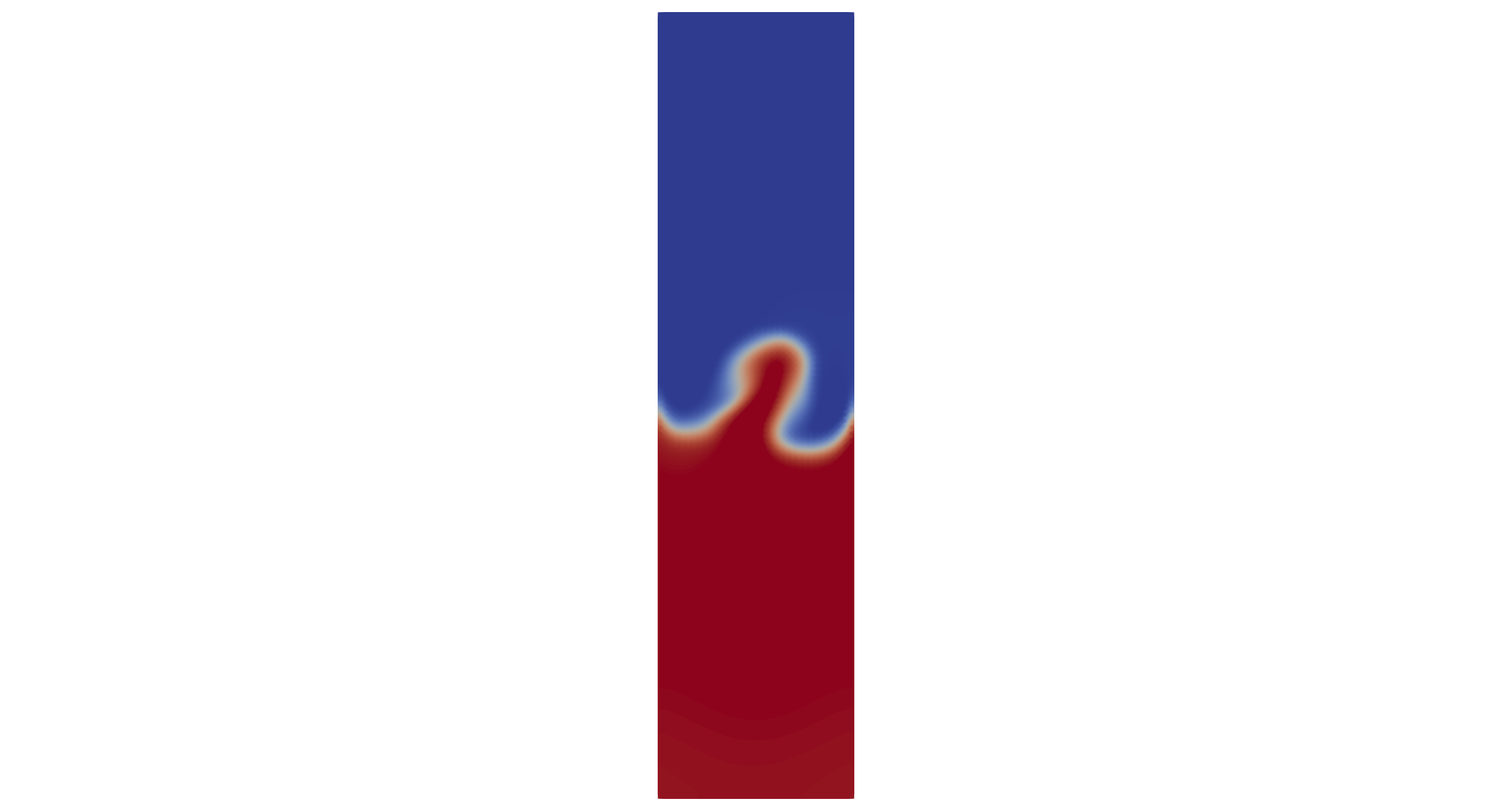}
\caption{Contours of the mass density at $t=0,1,2,3,4,5$ when $B_0=0.4$.}
\label{fig:B0p4}
\end{figure}

\begin{figure}
\centering
\plotsnapshots{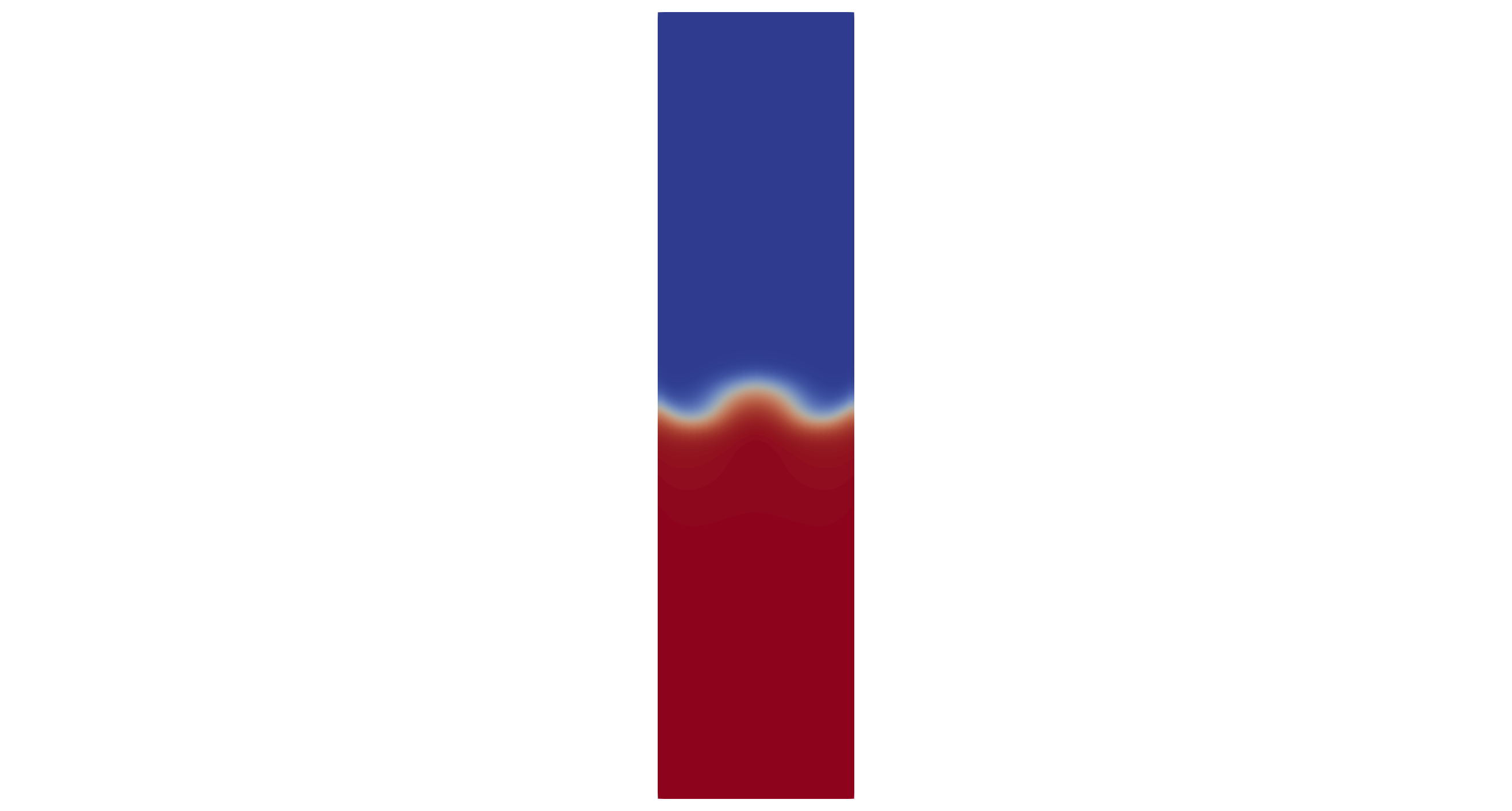}
\caption{Contours of the mass density at $t=0,1,2,3,4,5$ when $B_0=0.6$.}
\label{fig:B0p6}
\end{figure}

\begin{figure}
\centering
\plotsnapshots{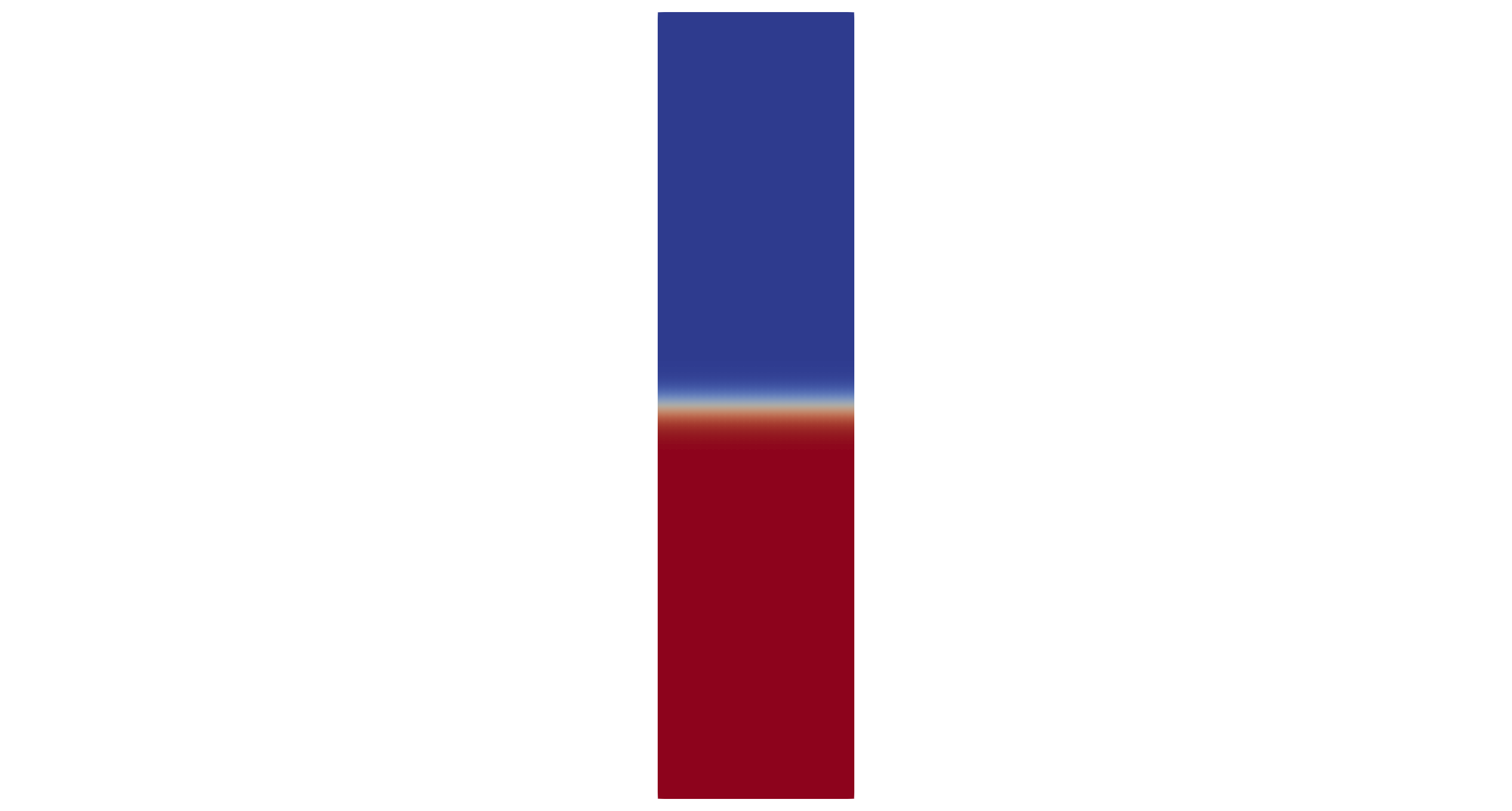}
\caption{Contours of the mass density at $t=0,1,2,3,4,5$ when $B_0=0.8$.}
\label{fig:B0p8}
\end{figure}

\begin{figure}
\centering
\hspace{-0.5in}
\begin{tikzpicture}
\begin{groupplot}[
group style={
    group name=my plots,
    group size=2 by 2,
    xlabels at=edge bottom,
    ylabels at=edge left,
    horizontal sep=2cm,vertical sep=2cm},
    ymode=log,xlabel=$t$,ylabel=$|F(t)-F(0)|$,
    ylabel style={yshift=-0.1cm},
    ymin = 10^-20,
    ymax = 1,
    legend style={at={(1.0,0.0)},
	anchor=south,
	column sep=1ex,
	text=black},
    legend columns=6
    ]
\nextgroupplot[legend to name=testLegend,title={$B_0=0.2$}]
\plotinvariants{Data/B0p2_invariants.dat}
\legend{$\int \rho \, {\rm d}x$ \\ $\int  [\frac{1}{2}  \rho |u|^2 + \frac{1}{2}|B|^2  + \epsilon(\rho, s) + \rho\phi  ]\, {\rm d}x$ \\$\int u \cdot B \, {\rm d}x$ \\ $\|\dv B\|_{L^2(\Omega)}$ \\}
\nextgroupplot[title={$B_0=0.4$}]
\plotinvariants{Data/B0p4_invariants.dat}
\nextgroupplot[title={$B_0=0.6$}]
\plotinvariants{Data/B0p6_invariants.dat}
\nextgroupplot[title={$B_0=0.8$}]
\plotinvariants{Data/B0p8_invariants.dat}
\end{groupplot}
\end{tikzpicture}
\ref{testLegend}
\caption{Evolution of mass, energy, cross-helicity, and $\|\dv B\|_{L^2(\Omega)}$ during simulations of the magnetic Rayleigh-Taylor instability with $B_0=0.2,0.4,0.6,0.8$.  The absolute deviations $|F(t)-F(0)|$ are plotted for each such quantity $F(t)$.}
\label{fig:invariants}
\end{figure}
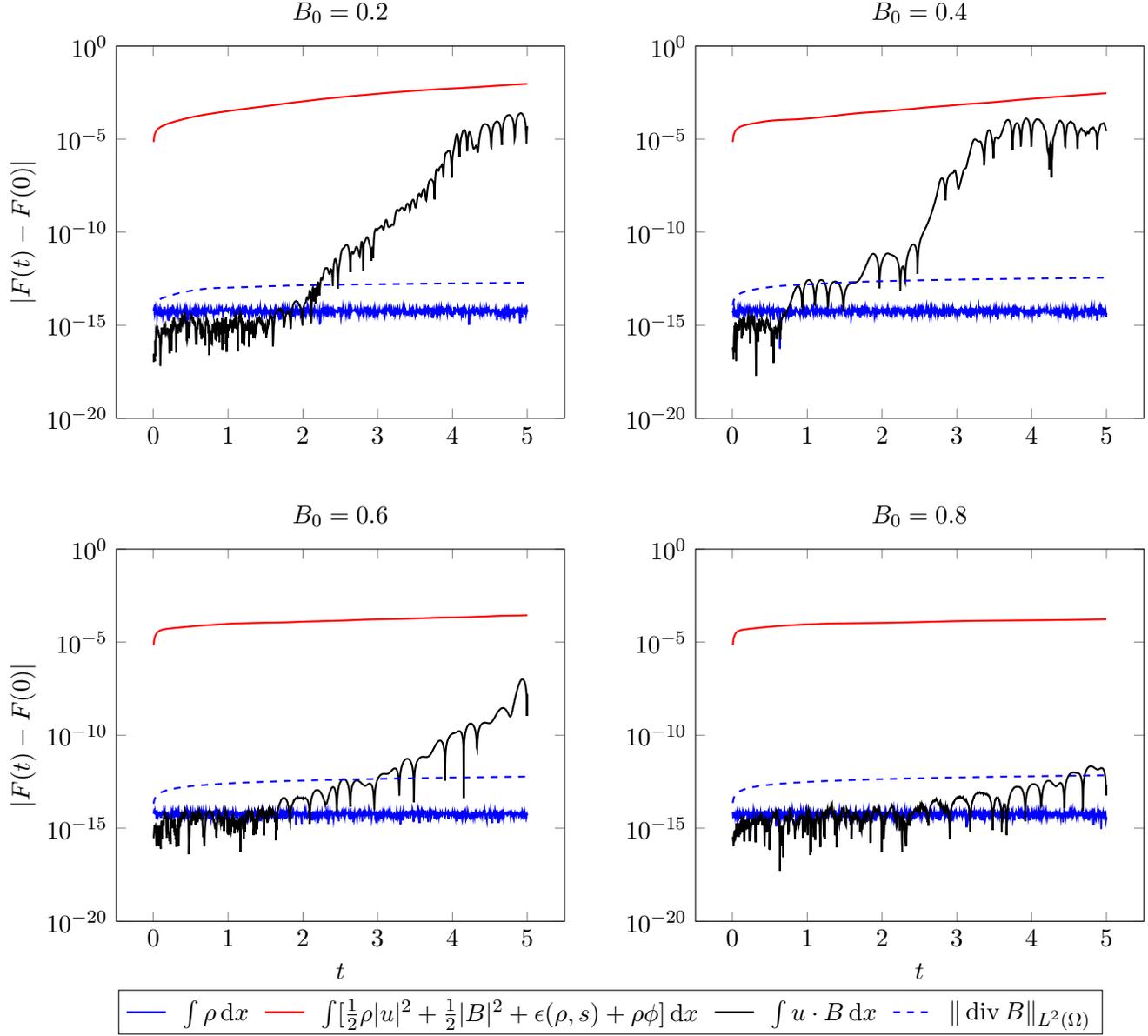

\appendix

\section{Lagrange-d'Alembert formulation of resistive MHD}\label{LdA}

In this appendix we explain how viscosity and resistivity can be included in the Lagrangian variational formulation by using the Lagrange-d'Alembert principle for forced systems. While viscosity can be quite easily included in the variational formulation by adding the corresponding virtual force term to the Euler-Poincar\'e principle, resistivity breaks the transport equation \eqref{B_frozen} hence the previous Euler-Poicar\'e approach for $B$ must be  appropriately modified.

We first observe that in absence of viscosity and resistivity, equations \eqref{EP_equations} can also be obtained from a variational formulation in which the variations $ \delta B$ are unconstrained. It suffices to consider, instead of \eqref{HP_MHD},  the variational principle
\begin{equation}\label{HP_modified} 
\delta \int_0^TL( \varphi , \partial _t \varphi , \varrho _0, \mathcal{B} ) - \left\langle \partial _t \mathcal{B} , \mathcal{C} \right\rangle  {\rm d} t=0,
\end{equation} 
with respect to arbitrary variations $ \delta \varphi $, $ \delta \mathcal{B} $, $ \delta \mathcal{C}$ with $ \delta \varphi $ and $ \delta \mathcal{B} $ vanishing at $t=0,T$. The second term in the action functional imposes $ \partial _t \mathcal{B} =0$, i.e., $ \mathcal{B} (t) =\mathcal{B} _0$. In Eulerian form, we get
\begin{equation}\label{EP_MHD_force_ideal} 
\delta \int_0^T\ell(u, \rho  , B)  - \left\langle \partial _t B -  \operatorname{curl}  (u \times B), C \right\rangle {\rm d}t =0
\end{equation} 
with constrained variations $\delta u = \partial _t v+ \pounds _uv$, $\delta \rho  = - \operatorname{div}( \rho  v)$ and free variations $ \delta B$, $ \delta C$ with $ v, \delta B$ vanishing at $t=0,T$. In \eqref{EP_MHD_force_ideal}, the magnetic field equation appears as a constraint with Lagrange multiplier $C$. The variational principle \eqref{EP_MHD_force_ideal} yields the three equations
\begin{align}
\partial _t \left( \frac{\delta \ell}{\delta u} + B \times  \operatorname{curl} C\right) + \pounds _u \left( \frac{\delta \ell}{\delta u} +B \times  \operatorname{curl} C  \right) & = \rho  \nabla \frac{\delta \ell}{\delta \rho  } \label{momentum_eq_2}\\
\partial_t B -     \operatorname{curl}  (u \times B)  &= 0 \label{B_eq_2}\\
\partial _t C + \operatorname{curl} C \times u + \frac{\delta \ell}{\delta B} &=0\label{C_eq_2}
\end{align}
which correspond to the variations associated to $v$, $ \delta B$, and $\delta  C$, respectively.
Using the formula
\[
\pounds _u ( B \times  \operatorname{curl}C ) = B \times \operatorname{curl}  ( \operatorname{curl}C \times u) +  \operatorname{curl}  (B \times u) \times   \operatorname{curl}C
\]
and \eqref{B_eq_2}--\eqref{C_eq_2} in the equations \eqref{momentum_eq_2} does yield \eqref{EP_equations}. 

\medskip 

Using the Lagrange-d'Alembert approach, the variational principle \eqref{HP_modified} can be modified as
\begin{equation}\label{HP_modified_force} 
\delta \int_0^TL( \varphi , \partial _t \varphi , \varrho _0, \mathcal{B} ) - \left\langle \partial _t \mathcal{B} , \mathcal{C} \right\rangle  {\rm d} t + \int_0^T  D(\varphi , \partial _t \varphi , \delta \varphi )+ E(\varphi , \mathcal{B} , \delta \mathcal{C})  {\rm d}t=0,
\end{equation}
for some expressions $D$ and $E$, bilinear in their last two arguments and invariant under the right action of $ \operatorname{Diff}( \Omega )$. In the Eulerian form, one gets
\begin{equation}\label{EP_MHD_force} 
\delta \int_0^T\ell(u, \rho  , B)  - \left\langle \partial _t B -  \operatorname{curl}  (u \times B), C \right\rangle {\rm d}t + \int_0^T  d( u,v)  + e(B, \delta C + \operatorname{curl} C \times v ){\rm d}t=0,\\
\end{equation} 
with $e$ and $f$ given by the expressions of $E$ and $F$ evaluated at $ \varphi = id$. To model viscosity and resistivity we choose \eqref{bilinear_form} and change the boundary condition of velocity to  $u|_{ \partial \Omega }=0$. This boundary condition corresponds in the Lagrangian description to the choice of the subgroup $ \operatorname{Diff}_0( \Omega ) $ of diffeomorphisms fixing the boundary pointwise. Application of \eqref{EP_MHD_force} yields the viscous and resistive barotropic MHD equations in the form 
\begin{align}
\left\langle \partial_t \frac{\delta \ell}{\delta u} , v \right\rangle + a\left(\frac{\delta \ell}{\delta u}, u, v\right)  + b\left( \frac{\delta \ell}{\delta \rho},\rho,v \right) + c\left(\frac{\delta \ell}{\delta B},B,v\right)&= d(u,v)  \label{velocity_res_NS2} \\
\langle \partial_t \rho, \sigma \rangle + b(\sigma,\rho,u) &=0\label{density_res_NS2} \\
\phantom{\int}\langle \partial_t B, C \rangle + c(C,B,u) &= e(B,C).\label{magnetic_res_NS2}
\end{align}

\end{document}